\documentclass[a4paper, reqno]{amsart}
\pdfoutput=1
\usepackage[english]{babel}
\usepackage[T1]{fontenc}
\usepackage{amssymb, upref}
\usepackage{enumerate, braket}
\usepackage[colorlinks=true]{hyperref}
\hypersetup{urlcolor=blue, citecolor=red}

\title[Spectrum and local eventual positivity]{Spectral properties of locally eventually positive operator semigroups}

\author[J. Mui]{Jonathan Mui}
\address{Jonathan Mui, School of Mathematics and Statistics, University of Sydney, NSW 2006, Australia}
\email{jonathan.mui@sydney.edu.au}


\date{\today}

\numberwithin{equation}{section}

\theoremstyle{plain}
\newtheorem{theorem}{Theorem}[section]
\newtheorem{proposition}[theorem]{Proposition}
\newtheorem{lemma}[theorem]{Lemma}
\newtheorem{corollary}[theorem]{Corollary}

\theoremstyle{definition}
\newtheorem{definition}[theorem]{Definition}
\newtheorem{example}[theorem]{Example}

\theoremstyle{remark}
\newtheorem{remark}[theorem]{Remark}

\DeclareMathOperator{\rg}{rg}
\DeclareMathOperator{\dist}{dist}
\DeclareMathOperator{\codim}{codim}
\DeclareMathOperator{\abs}{abs}
\DeclareMathOperator{\hol}{hol}

\let\Re\relax
\DeclareMathOperator{\Re}{Re}

\newcommand{\NN}{\mathbb{N}}
\newcommand{\ZZ}{\mathbb{Z}}
\newcommand{\RR}{\mathbb{R}}
\newcommand{\CC}{\mathbb{C}}

\begin{document}
	
\begin{abstract}
	This paper considers strongly continuous semigroups of operators on Banach lattices which are \emph{locally eventually positive}, a property that was first investigated in the context of concrete fourth-order evolution equations. We construct a simple example to show that the typical assumptions on the spectrum of the semigroup generator considered currently in the literature are far from necessary in the more general setting of local eventual positivity. Under minimal additional assumptions, we obtain results on the asymptotic behaviour of orbits, as well as necessary conditions on the peripheral point spectrum of locally eventually positive semigroups.
\end{abstract}

\maketitle

\section{Introduction}

\subsection*{Statement of main results}
Semigroups of linear operators have long been used to study linear evolution equations in an abstract functional analytical setting. When the underlying space is an ordered function space, then the notion of a positive semigroup is relevant. Today, much is known about stability and asymptotic behaviour of positive semigroups, see for example the classic monograph~\cite{AGG} or the more recent text~\cite{BFR}. By contrast, the study of \emph{eventually positive semigroups} --- in which, roughly speaking, positivity occurs only for sufficiently large times --- is still relatively young. Eventual positivity for matrix semigroups was studied in~\cite{NT}, and since then has developed into a sub-field of its own right, as a natural extension of Perron-Frobenius theory. Motivated by the case study~\cite{Dan}, a systematic theory of eventually positive semigroups and resolvents on infinite-dimensional Banach lattices was developed by Daners, Gl\"{u}ck and Kennedy in~\cite{DGK1,DGK2}. Within these papers, the reader will already find various applications to partial differential equations. For more recent developments, we point out in particular~\cite{DG18,ArG21} for theoretical results, and~\cite{DKP,GM} for applications to fourth-order differential operators.

However, it was discovered in~\cite{GG-lep} that a still weaker positivity property was required to describe the asymptotic behaviour of solutions to the \emph{biharmonic heat equation} $u_t + (-\Delta)^2 u=0$ on $\RR^N$. This phenomenon, now known as \emph{local eventual positivity}, has been investigated in many other concrete PDE problems, and is the main subject of this article. The intuitive idea is that the solution becomes eventually positive on every compact subset of $\RR^N$, but not globally. Based on this idea, we formulate an abstract notion of a locally eventually positive semigroup on a Banach lattice, and derive some spectral properties under quite general conditions. We give a taste of some key results below, without yet stating the technical definitions.
\begin{theorem}
	\label{thm:LEP-intro}
	Let $(e^{tA})_{t\ge 0}$ be a real $C_0$-semigroup on a complex Banach lattice $E$. Assume that the spectral bound $s(A)$ of the generator satisfies $s(A)>-\infty$.
	\begin{enumerate}[\upshape(i)]
		\item \emph{(Asymptotic positivity)} If the semigroup is locally eventually positive and the orbits are relatively compact, then for each $f\in E$ it holds that $\dist(e^{tA}f, E_+) \to 0$, where $E_+$ is the positive cone of $E$.
		\item \emph{(Triviality of peripheral point spectrum)} If the semigroup is strongly locally eventually positive and the orbits of the rescaled semigroup $(e^{t(A-s(A))})_{t\ge 0}$ are relatively weakly compact, then the intersection of the point spectrum of the generator $\sigma_p(A)$ with the line $s(A)+i\RR$ is contained in $\{s(A)\}$.
	\end{enumerate}
\end{theorem}
The precise notion of (strong) local eventual positivity will be introduced in Definition~\ref{def:abstract-LEP}. Statement (i) is proved in Theorem~\ref{thm:asymp-pos}, and Statement (ii) is the content of Theorem~\ref{thm:strong-pos}. It should be noted that the above results do not require \emph{a priori} assumptions on the spectrum of the semigroup generator, in contrast to existing results on local eventual positivity in the abstract Banach lattice setting. We will discuss this aspect further below.

\subsection*{Organisation of the article}
At the end of this introductory section, the reader will find a brief overview of known results on local eventual positivity, which gives some context to the results of this paper. Section~\ref{sec:example} concerns the analysis of a particular locally eventually positive semigroup which stands apart from the examples studied in the current literature, as it shows that the usual sufficient conditions on the spectrum to obtain local eventual positivity are far from necessary.

In Section~\ref{sec:abstract-stuff}, we formulate an abstract notion of local eventual positivity on Banach lattices (Definition~\ref{def:abstract-LEP}) and study the spectrum of generators of locally eventually positive semigroups under mild regularity assumptions on the semigroup. We prove the results announced in Theorem~\ref{thm:LEP-intro}, and also give sufficient conditions for the generator of a locally eventually positive semigroup to have cyclic peripheral point spectrum (Corollary~\ref{cor:cyclic}). Finally, in Section~\ref{sec:stability} we investigate the question of whether the spectral bound $s(A)$ of the generator of a locally eventually positive semigroup is a spectral value. We obtain an affirmative answer in Theorem~\ref{thm:spectral-bound} under quite general conditions.

\subsection*{Notations and terminology}
We assume the reader is familiar with $C_0$-semigroups and basic notions in the theory of vector lattices (also known as Riesz spaces). Let us fix some conventions that will be used throughout the article.

We use the standard notations $\vee, \wedge$ to denote respectively the supremum and infimum in a vector lattice. A normed vector lattice is a vector lattice equipped with a \emph{lattice norm}, which means that if $|x|\le |y|$, then $\|x\|\le \|y\|$. A \emph{Banach lattice} is, of course, a complete normed vector lattice. By definition, a \emph{complex Banach lattice} is the complexification of a Banach lattice over $\RR$ --- see for example~\cite[II Section 11]{Sch}. If $E$ is a complex Banach lattice, then the underlying Banach space over $\RR$ is called the \emph{real part} of $E$ and denoted by $E_\RR$.

For Banach lattices $E, F$, a linear operator $T:E\to F$ is called \emph{positive} if $Tf\ge 0$ in $F$ for all $f\ge 0$ in $E$. We write $E_+$ and $F_+$ for the positive cones in $E$ and $F$ respectively; thus an operator $T:E\to F$ is positive if and only if $T(E_+)\subseteq F_+$. If $E, F$ are complex Banach lattices, then a bounded linear operator $T:E\to F$ is called \emph{real} if $T(E_\RR) \subseteq F_\RR$. Note that all positive operators are real.

If $A:D(A)\subseteq E\to E$ is a closed, densely defined linear operator on a Banach space $E$, we write $(e^{tA})_{t\ge 0}$ for the semigroup generated by $A$. The \emph{spectral bound} of the generator is the quantity
\begin{equation*}
	s(A) := \sup\{\Re\lambda : \lambda\in\sigma(A)\} \in [-\infty,\infty],
\end{equation*}
with the convention that $s(A)=-\infty$ if $\sigma(A)=\emptyset$.

\subsection*{Recent developments in local eventual positivity}
\label{sec:LEP-background}
In~\cite{GG-lep}, the following phenomenon was observed in solutions to the biharmonic heat equation $u_t + (-\Delta)^2 u=0$ on $\RR^N$: for every non-zero, continuous and compactly supported initial datum $u_0$ with $u_0\ge 0$, for every compact subset $K\subset\RR^N$, there exists a time $\tau>0$ (depending on $K$ and $u_0$) such that the corresponding solution $u=u(t,x)$ satisfies
\begin{equation}
	\label{eq:LEP-concrete}
	u(t,x) > 0 \quad\text{for all } x \in K\text{ and } t\ge\tau.
\end{equation}
In addition, the authors showed that there does \emph{not} exist any $\tau\in (0,\infty)$ for which $u(t,x)>0$ for all $x\in\RR^N$ and all $t\ge\tau$. More informally, eventual positivity of the solution does not hold `globally'. These features justify the term `local eventual positivity'.

For the biharmonic operator $(-\Delta)^2$ on $\RR^N$, local eventual positivity was studied further in~\cite{FGG}, and later in~\cite{FF-lep}, it was shown that solutions to polyharmonic heat equations $u_t + (-\Delta)^\alpha u=0$, for all $\alpha>1$, satisfy~\eqref{eq:LEP-concrete} as well. In the recent preprint~\cite{DGM}, the biharmonic heat equation is studied on infinite cylinders $\RR\times\Omega$ (with $\Omega$ a bounded domain in $\RR^N$ with smooth boundary), where it is shown that local eventual positivity for the evolution equation depends on positivity properties for the corresponding elliptic problem on the cross-section domain $\Omega$. On the other hand, by building on the ideas of~\cite{DGK2,DG18}, an operator-theoretic approach to studying local eventual positivity was initiated by Arora in~\cite{Ar21}. The results on (locally) eventually positive semigroups within the abstract framework have thus far required rather strong assumptions on the spectrum of the generator. In particular, the existence of isolated spectral values is essential, since spectral projections are used extensively. From the point of view of differential equations, this theory is therefore very effective for problems on bounded domains, or on unbounded domains with probability measures --- see~\cite[Section 3.2]{AGRT} for a recent application in the latter setting. 

On the other hand, many examples on unbounded domains do not fall within the scope of the current operator-theoretic setting, such as polyharmonic heat equations on Euclidean space (with the Lebesgue measure). In light of the discussion above, for the further development of the general theory of (local) eventual positivity, it is therefore natural to investigate what can be done if the `usual' spectral assumptions are dropped. The present article represents a step in this direction.

\section{An explicit example}
\label{sec:example}
As we recalled in the introduction, the study of (local and global) eventual positivity of semigroups thus far has relied on rather strong spectral assumptions, such as requiring the spectral bound $s(A)$ of the generator to be a pole of the resolvent operator. By contrast, in this section we show that the generator of a locally eventually positive semigroup may have quite pathological spectrum. The main feature of our example is a distinct `lack of compactness'.

\begin{example}
	\label{exam:LEP-semigroup}
	Consider the Banach lattice $E = L^1(\RR) \cap C_0(\RR)$ with norm $\|f\|_E = \max\{\|f\|_1, \|f\|_\infty\}$. There exists a real $C_0$-semigroup $(T_t)_{t\ge 0}$ on $E$ such that, for every $f\in E_+\setminus\{0\}$ and every compact interval $J\subset\RR$, there exists $\tau\ge 0$ such that
	\begin{equation*}
		(T_t f)(x)>0 \quad\text{for all } x\in J\text{ and } t\ge\tau.
	\end{equation*}
	In short, $(T_t)_{t\ge 0}$ is locally eventually positive in the sense of property~\eqref{eq:LEP-concrete}. Moreover, the generator $A$ of the semigroup satisfies $\sigma(A)=i\RR$.
		
	We give the explicit construction now and postpone the discussion of the spectrum of the generator to Proposition~\ref{prop:example-LEP-semigroup}. Consider the functional
	\begin{equation*}
		\varphi(f) := \int_\RR f(x) \,dx \qquad \text{for all } f \in E.
	\end{equation*}
	Clearly $|\varphi(f)| \le \|f\|_1 \le \|f\|_E$, so $\varphi \in E'$. Hence $G:=\ker\varphi$ is a closed subspace of $E$, and it is obviously translation invariant. Now fix $f_0 \in E$ such that $0<f_0(x) \le 1$ for all $x \in \RR$ and $\int_\RR f_0(x) \,dx = 1$ (for example, one may choose the standard Gaussian $f_0(x)=(2\pi)^{-1/2}e^{-x^2/2}$). Consider the rank-one projection $Pf := (\varphi\otimes f_0)(f) = \varphi(f)f_0$, and observe that $\ker P =\ker\varphi$. Then every $f \in E$ can be expressed as
	\begin{equation*}
		f = Pf - (I-P)f = \varphi(f)f_0 + [f-\varphi(f)f_0],
	\end{equation*}
	and we obtain the direct sum decomposition
	\begin{equation*}
		E = F\oplus G,\qquad \text{where }F=\rg P\text{ and }G=\ker P.
	\end{equation*}
	
	Let $(S_t)_{t\ge 0}$ denote the left translation semigroup on $E$, i.e.\ $(S_t f)(x) := f(x+t)$ for all $f\in E, t\ge 0$. We define a semigroup on $E$ by $T_t := I_F \oplus \tau_t |_{G}$, which has the explicit formula
	\begin{equation}
		\label{eq:LEP-semigroup}
		(T_t f)(x) = \varphi(f)f_0(x) + [f(x+t) - \varphi(f) f_0(x+t)] \qquad t \ge 0, x\in\RR
	\end{equation}
	for all $f \in E$. To see that~\eqref{eq:LEP-concrete} holds, let $f \in E_+ \setminus\{0\}$ be arbitrary, and fix a compact interval $[-R, R]$. Since $f_0(x) > 0$ for all $x \in \RR$ and $f_0$ is continuous on the compact set $[-R, R]$, there exists $\delta=\delta(R) > 0$ such that $f_0(x) \ge\delta$ for all $x \in [-R, R]$. On the other hand, $f_0 \in C_0(\RR)$ implies that there exists $M > 0$ such that $0 < f_0(x) < \delta$ whenever $|x| \ge M$. For all $t > 0$ sufficiently large such that $|x+t| \ge M$, it follows that $0 < f_0(x+t) < \delta$, and therefore there exists $\tau > 0$ such that $f_0(x) - f_0(x+t) > 0$ for all $x \in [-R, R]$ and all $t\ge\tau$. We find
	\begin{equation}
		\label{eq:LEP-uniform}
		(T_t f)(x) = \varphi(f)[f_0(x) - f_0(x+t)] + f(x+t) > 0
	\end{equation}
	for all $x \in [-R,R]$ and all $t \ge \tau$. 
	
	As an additional observation, formula~\eqref{eq:LEP-uniform} shows that the time to positivity can be chosen to depend only on the difference $f_0(x)-f_0(x+t)$, since $\varphi(f)>0$ and $f(x+t) \ge 0$ for all $x\in\RR, t\ge 0$ and $f\in E_+\setminus\{0\}$. Thus the semigroup can in fact be called \emph{uniformly} locally eventually positive.
\end{example}

We now show that the generator of the semigroup constructed above has spectrum equal to the imaginary axis.
\begin{proposition}
\label{prop:example-LEP-semigroup}
	The generator $A$ of the semigroup defined by~\eqref{eq:LEP-semigroup} is given by
	\begin{equation*}
		A = 0_{|F} \oplus B_{|G}, \qquad D(A) = F \oplus D(B_{|G})
	\end{equation*}
	where $B_{|G}$ generates the left translation semigroup on $G$ and is thus given by $Bf = \frac{df}{dx}$ with
	\begin{equation*}
		D(B_{|G}) = \left\{ f \in G : \frac{df}{dx} \in G \right\}.
	\end{equation*}
	The spectrum of $A$ is $\sigma(A) = i\RR$, and 0 is an eigenvalue with a one-dimensional eigenspace spanned by the positive function $f_0$.
\end{proposition}

\begin{proof}
	The description of the generator follows from standard facts about subspace semigroups and generators of translation semigroups on $C_0(\RR)$, see for example~\cite[Chapter II, 2.2.3 \& 2.2.10]{EN00}. Since $|\varphi(f)| \le \|f\|_1$, and $\|f_0\|_\infty \le 1$ and $\|f_0\|_1 = 1$, from formula~\eqref{eq:LEP-semigroup} we can derive the simple estimates
	\begin{align*}
		|T_tf(x)| &\le 2\|f\|_1 + \|f\|_\infty \quad\text{for all }x \in \RR, \text{and} \\
		\|T_t f\|_1 &\le 3\|f\|_1 \quad\text{for all } t \ge 0.
	\end{align*}
	Hence
	\begin{equation*}
		\|T_t f\|_E \le 3\|f\|_E
	\end{equation*}
	which shows that $(T_t)_{t\ge 0}$ is a uniformly bounded semigroup. It is also clear from~\eqref{eq:LEP-semigroup} that $(T_t)_{t\ge 0}$ extends to a bounded $C_0$-group on $E$, since this is true of the translation semigroups. Hence we deduce that $\sigma(A) \subseteq i\RR$.
	
	Now we show that every $\lambda\in i\RR \setminus\{0\}$ is an approximate eigenvalue of $A$. Consider the case $\lambda=i\alpha$ with $\alpha>0$. We claim that
	\begin{equation}
		w_n(x) := \frac{1}{\sqrt{n}}\left(1+\frac{ix}{n}\right) e^{-\frac{\alpha x^2}{2n} + i\alpha x}, \qquad x\in\RR, n\in\NN
	\end{equation}
	forms a sequence of approximate eigenvectors for the spectral value $i\alpha$. Firstly we observe that
	\begin{equation*}
		|w_n(x)| = \frac{1}{\sqrt{n}} \sqrt{1+\tfrac{x^2}{n^2}} e^{-\frac{\alpha x^2}{2n}},
	\end{equation*}
	and thus
	\begin{equation*}
		\|w_n\|_1 \ge \frac{1}{\sqrt{n}}\int_\RR e^{-\frac{\alpha x^2}{2n}} \,dx = \sqrt{\frac{2\pi}{\alpha}}.
	\end{equation*}
	On the other hand, it is easily seen that $\|w_n\|_\infty \to 0$ as $n\to\infty$, and using the elementary inequality $(a+b)^{1/2} \le a^{1/2} + b^{1/2}$ for all $a,b\ge 0$, we estimate
	\begin{align*}
		\int_\RR |w_n(x)| \,dx &\le \frac{1}{\sqrt{n}}\int_\RR e^{-\frac{\alpha x^2}{2n}} \,dx + \frac{1}{n\sqrt{n}}\int_\RR |x| e^{-\frac{\alpha x^2}{2n}}\,dx \\
		&= \sqrt{\frac{2\pi}{\alpha}} + \frac{1}{\sqrt{n}} \int_\RR |y| e^{-\frac{\alpha y^2}{2}} \,dy \\
		&\le \sqrt{\frac{2\pi}{\alpha}} + \int_\RR |y| e^{-\frac{\alpha y^2}{2}} \,dy,
	\end{align*}
	which shows that $\sup_{n\in\NN}\|w_n\|_1 <\infty$. Hence there exists a constant $C>0$ such that
	\begin{equation*}
		\sqrt{\frac{2\pi}{\alpha}} \le \|w_n\|_1 + \|w_n\|_\infty = \|w_n\|_E \le C
	\end{equation*}
	for all $n\in\NN$. Thus $(w_n)_{n\in\NN}$ is a bounded sequence in $E$ with norms uniformly bounded away from 0. Clearly each $w_n$ is a smooth function. To show that $w_n\in D(A)$, it remains to check that $\int_\RR w_n \,dx = 0$ for all $n\in\NN$. Observe that
	\begin{equation*}
		 w_n = (u_n+iv_n)e^{i\alpha(\cdot)}, \quad\text{where } u_n(x)=\frac{1}{\sqrt{n}} e^{-\frac{\alpha x^2}{2n}}, \quad v_n(x)=\frac{x}{n\sqrt{n}}e^{-\frac{\alpha x^2}{2n}}.
	\end{equation*}
	Then $v_n = -(1/\alpha)u_n'$, and integration by parts reveals that
	\begin{equation}
		\label{eq:un-vn-ibp}
		\int_\RR u_n(x)\cos(\alpha x) \,dx = -\frac{1}{\alpha} \int_\RR u_n'(x) \sin(\alpha x) \,dx = \int_\RR v_n(x)\sin(\alpha x)\,dx.
	\end{equation}
	By considering odd and even symmetries together with~\eqref{eq:un-vn-ibp}, we find
	\begin{align*}
		\int_\RR w_n(x) \,dx &= \int_\RR (u_n(x) + iv_n(x))(\cos(\alpha x)+i\sin(\alpha x)) \,dx \\
		&= \int_\RR [u_n(x)\cos(\alpha x) - v_n(x)\sin(\alpha x)] + i[v_n(x)\cos(\alpha x)+u_n(x)\sin(\alpha x)] \,dx \\
		&= 0
	\end{align*}
	and hence $w_n \in D(A)$ for all $n\in\NN$. Finally, an easy computation yields
	\begin{equation*}
		Aw_n - i\alpha w_n = w_n' - i\alpha w_n = \frac{1}{n\sqrt{n}}\left(1-\frac{\alpha x}{n}\right) e^{-\frac{\alpha x^2}{2n}+i\alpha x}.
	\end{equation*}
	Clearly $\|Aw_n - i\alpha w_n\|_\infty \le C'n^{-3/2} \to 0$ as $n\to\infty$, for an appropriate constant $C'>0$. Then
	\begin{align*}
		\|Aw_n - i\alpha w_n\|_1 &\le \frac{1}{n\sqrt{n}}\int_\RR e^{-\frac{\alpha x^2}{2n}} \,dx + \frac{\alpha}{n^2\sqrt{n}}\int_\RR |x| e^{-\frac{\alpha x^2}{2n}} \,dx \\
		&= \frac{1}{n}\sqrt{\frac{2\pi}{\alpha n}} + \frac{2\alpha}{n^2\sqrt{n}}\int_0^\infty xe^{-\frac{\alpha x^2}{2n}} \,dx \\
		&= \frac{1}{n}\sqrt{\frac{2\pi}{\alpha n}} + \frac{2}{n\sqrt{n}} \longrightarrow 0
	\end{align*}
	as $n\to\infty$. We conclude that $\|Aw_n - i\alpha w_n\|_E \to 0$ as $n\to\infty$, and hence $(w_n)_{n\in\NN}$ is a sequence of approximate eigenvectors corresponding to $\lambda=i\alpha$. If $\alpha<0$, we simply write $0<\beta:=-\alpha$, and apply the above construction with $\beta$ in place of $\alpha$. Finally, it is obvious that $Af_0 = 0$, so $F = \mathrm{span}(f_0)$ is the one-dimensional eigenspace corresponding to the eigenvalue 0.
\end{proof}

\begin{remark}
	Note that the spectral bound $s(A)$ of the semigroup in Example~\ref{exam:LEP-semigroup} is $0$, and is an eigenvalue of the generator $A$. Proposition~\ref{prop:example-LEP-semigroup} shows that $0$ is not a pole of the resolvent, since it is not an isolated spectral value. Nevertheless, the associated eigenspace is one-dimensional and spanned by the strictly positive function $f_0$. One may be tempted to apply~\cite[Theorem 3.3]{Ar21} to deduce at least individual local eventual positivity, but this is not possible since the semigroup~\eqref{eq:LEP-semigroup} does not converge strongly on $E$ as $t\to\infty$.
\end{remark}

\section{General results on the spectrum of locally eventually positive semigroups}
\label{sec:abstract-stuff}

\subsection{Lattice homomorphisms and band projections}
In this section, we present some consequences of local eventual positivity under minimal assumptions on the spectrum of the generator. Recall that in a vector lattice $E$, two elements $x,y$ are \emph{disjoint} if $|x|\wedge |y|=0$, and we write $x \perp y$. Given a non-empty subset $A\subset E$, the \emph{disjoint complement} $A^{\mathrm{d}}$ of $A$ consists of all $x\in E$ for which $x\perp y$ for all $y\in A$. A \emph{band} $B$ is an ideal that is \emph{order closed}, that is, if $D\subset B$ is a subset such that $\sup D$ exists in $E$, then $\sup D \in B$.

We will also require some special types of positive operators. If there is an ideal $A$ of $E$ such that $E=A\oplus B$, then $B$ is called a \emph{projection band}, and in this case, it holds that $A=B^\mathrm{d}$. We may then define a positive projection $P : E\to E$ such that $\rg P = B$. We call $P$ the \emph{band projection} (or \emph{order projection} in some texts) onto $B$, and the complementary projection $I-P$ satisfies $\rg (I-P)=B^{\mathrm{d}}$. It is known that a projection $P$ on a vector lattice $E$ is a band projection if and only if $0 \le P \le I$, where $I$ is the identity operator on $E$. Moreover, if $A,B$ are bands in $E$ with corresponding band projections $P_A, P_B$, then
\begin{equation}
\label{eq:band-proj}
	A\subseteq B \iff P_A \le P_B \iff P_A P_B = P_B P_A = P_A.
\end{equation}
If $P$ is the band projection onto a band $B$, then note that $I-P$ is the band projection onto the disjoint complement $B^{\mathrm{d}}$. Proofs for these standard facts may be found, for example, in~\cite[Chapter 1, Section 3]{AB}.

Next, let us recall that for vector lattices $E,F$ over $\RR$, a \emph{lattice homomorphism} is a linear operator $T:E\to F$ such that $T$ preserves the lattice operations: $T(x\vee_E y) = Tx \vee_F Ty$. It is simple to check that $T$ is positive, and that the defining property is equivalent to the condition $|Tx|_F = T|x|_E$. Another useful characterisation is that $T$ is a lattice homomorphism if and only if $T$ is positive and \emph{disjointness preserving}, that is,
\begin{equation}
\label{eq:disjoint-preserve}
	x\wedge_E y = 0 \Rightarrow Tx \wedge_F Ty=0,
\end{equation}
cf.~\cite[Theorem 2.14]{AB}. It follows that every band projection is a lattice homomorphism, since if $x\wedge y=0$ in $E$, then $0\le Px \le x$ and $0\le Py \le y$ imply $0\le Px \wedge Py \le x\wedge y=0$, so $Px \wedge Py=0$.

\subsection{Local eventual and asymptotic positivity}
By way of motivation, let us recall a situation described in~\cite[Section 7]{Ar21}. If $(K_n)_{n\in\NN}$ is an increasing sequence (i.e.\ $K_n \subseteq K_{n+1}$ for all $n\in\NN$) of bands in $E$ such that $\bigcup_{n\in\NN}K_n = E$, and if $P_n$ denotes the band projection onto $K_n$, then by property~\eqref{eq:band-proj}, we have $0\le P_n \le P_{n+1} \le I$ for all $n\in\NN$. Thus we speak of an \emph{increasing sequence} of band projections $(P_n)_{n\in\NN}$. Each projection $P_n$ localises a function $f$ onto the compact subset $K_n$. Given a semigroup $(e^{tA})_{t\ge 0}$, we may then ask if each localised orbit $t\mapsto P_n e^{tA}f$ becomes eventually positive.

With the concepts introduced in the previous subsection, we are ready to formulate an abstract definition of local eventual positivity for $C_0$-semigroups on Banach lattices. Recall also that in a normed vector lattice $E$, a positive element $u\in E$ is called \emph{quasi-interior} if the principal ideal $E_u$ (i.e.\ the ideal generated by $u$) is dense in $E$. 
\begin{definition}
\label{def:abstract-LEP}
	Let $E$ be a complex Banach lattice, and let there be given an increasing sequence of positive operators $(P_n)_{n\in\NN}$ on $E$ converging strongly to the identity. A $C_0$-semigroup $(e^{tA})_{t\ge 0}$ on $E$ with generator $A$ is said to have the \emph{LEP property} with respect to the sequence $(P_n)_{n\in\NN}$ if for every $n\in\NN$ and $f\in E_+\setminus\{0\}$, there exists $\tau>0$ (depending on $n$ and $f$) such that
	\begin{equation}
		\label{eq:LEP}\tag{LEP}
		P_n e^{tA}f \ge 0 \quad\text{for all } t\ge\tau.
	\end{equation}
	If in addition, for all $t\ge\tau$, $P_n e^{tA}f$ is even a quasi-interior point of the closed sublattice generated by $P_n E$ in $E$, we say that the semigroup has the \emph{strong LEP property} with respect to the sequence $(P_n)_{n\in\NN}$.
\end{definition}

\begin{remark}
	The abbreviation LEP stands for \emph{local eventual positivity}. Example~\ref{ex:abstract-LEP} below shows that Definition~\ref{def:abstract-LEP} is a natural abstraction of the property~\eqref{eq:LEP-concrete}.
	
	In Definition~\ref{def:abstract-LEP}, we have admitted the more general notion of an increasing sequence of positive operators. Nevertheless, in many applications, the sequence $(P_n)_{n\in\NN}$ consists of band projections --- we refer the reader~\cite[Section 5]{Ar21} for some examples. On the other hand, note that some commonly occurring Banach lattices do not have non-trivial band projections. This is the case for $E=C(K)$ where $K$ is a compact, connected Hausdorff space (see~\cite[Example 5, p.\ 63]{Sch}).
\end{remark}

\begin{example}
\label{ex:abstract-LEP}
	We give two concrete examples in which the assumptions of Definition~\ref{def:abstract-LEP} hold.
	\begin{enumerate}[(i)]
	\item Let $(\Omega,\mu)$ be a $\sigma$-finite measure space, and let $(K_n)_{n\in\NN}$ be an increasing sequence of measurable subsets of $\Omega$ such that $\bigcup_{n\in\NN}K_n = \Omega$. We consider the Banach lattice $E=L^p(\Omega,\mu)$ for $1\le p<\infty$. It is easy to see that the set of $f\in E$ with support in $K_n$ is a band. For each $n\in\NN$, let $\chi_n$ denote the indicator function of the set $K_n$, and define the band projection $P_n f := f \chi_n$. Then by the dominated convergence theorem, $P_n$ converges strongly to the identity. If a $C_0$-semigroup $(e^{tA})_{t\ge 0}$ on $E$ has the~\eqref{eq:LEP} property with respect to $(P_n)_{n\in\NN}$, this means that for every $n\in\NN$ and $f\in E_+\setminus\{0\}$, there exists $\tau>0$ such that $(e^{tA}f)(x) \ge 0$ for a.e.\ $x\in K_n$, for all $t\ge\tau$.
	
	\item Let $\Omega$ be a locally compact Hausdorff space, and assume there exists an increasing sequence $(K_n)_{n\ge 0}$ of compact subsets such that $\bigcup_{n\ge 0}K_n=\Omega$. For each $n\ge 1$, let $\varphi_n$ be a Urysohn function for the pair of closed sets $K_{n-1}$ and $\overline{\Omega\setminus K_n}$: $\varphi_n$ is continuous on $\Omega$ with $0\le \varphi_n \le 1$, $\varphi_n \equiv 1$ on $K_{n-1}$ and $\varphi_n \equiv 0$ on $\overline{\Omega\setminus K_n}$. On the Banach lattice $C_0(\Omega)$, one verifies easily that the operators $Q_n f :=\varphi_n f$ are lattice homomorphisms, and the sequence $Q_n$ increases strongly to the identity operator.
	\end{enumerate}
\end{example}

Our first result shows that if a semigroup has the~\eqref{eq:LEP} property, then under mild regularity conditions, it is already `asymptotically positive' in the sense of~\eqref{eq:asymp-pos} below. We remark that the notion of \emph{individual asymptotic positivity} for $C_0$-semigroups was defined in~\cite[Definition 8.1]{DGK2} under the additional assumptions $s(A)>-\infty$ and boundedness of the rescaled semigroup $(e^{t(A-s(A))})_{t\ge 0}$. Given $x\in E$, the \emph{orbit} of $x$ under the semigroup $(e^{tA})_{t\ge 0}$ is the set $\{e^{tA}x : t \ge 0\}$.
\begin{theorem}
\label{thm:asymp-pos}
	Let $E$ be a complex Banach lattice, and let there be given a sequence $(P_n)_{n\in\NN}$ of positive operators on $E$ increasing strongly to the identity. Let $(e^{tA})_{t\ge 0}$ be a $C_0$-semigroup on $E$. If~\eqref{eq:LEP} is satisfied with respect to the sequence $(P_n)_{n\in\NN}$, and the semigroup has relatively compact orbits, then for every $x\in E_+$, it holds that
	\begin{equation}
	\label{eq:asymp-pos}
		\dist(e^{tA}x, E_+) \longrightarrow 0 \quad\text{as } t\to\infty.
	\end{equation}
\end{theorem}

\begin{proof}
	Fix $x\in E_+\setminus\{0\}$. It is an elementary fact that if a sequence of bounded linear operators on a Banach space converges strongly, then the convergence is uniform on relatively compact subsets (see e.g.~\cite[III Lemma 3.7]{Kato}). Since $P_n$ converges strongly to the identity $I$, and the set $\{e^{tA}x : t\ge 0\}$ is relatively compact in $E$, we obtain
	\begin{equation}
	\label{eq:outside-Pn-estimate}
		\lim_{n\to\infty}\sup_{t\ge 0}\|(I-P_n)e^{tA}x\| = 0.
	\end{equation}
	By the~\eqref{eq:LEP} property, for each $n$ there exists $t_n \ge 0$ such that $P_n e^{tA}x \ge 0$ for all $t\ge t_n$, and the sequence $(t_n) \subset [0,\infty)$ may be chosen to be strictly increasing. For each $n\in\NN$, we have
	\begin{equation*}
		\dist(e^{tA}x,E_+) \le \dist(P_n e^{tA}x,E_+) + \dist((I-P_n)e^{tA}x,E_+) \le \|(I-P_n)e^{tA}x\|
	\end{equation*}
	for all $t\ge t_n$. Now~\eqref{eq:outside-Pn-estimate} shows that
	\begin{equation*}
		\sup_{t\ge t_n} \dist(e^{tA}x,E_+) \le \sup_{t\ge 0}\|(I-P_n)e^{tA}x\| \longrightarrow 0
	\end{equation*}
	as $n\to\infty$, which yields $\limsup_{t\to\infty} \dist(e^{tA}x,E_+)=0$. This proves the convergence~\eqref{eq:asymp-pos}.
\end{proof}

We give a simple case in which the assumptions of the above theorem hold.
\begin{corollary}
	Let $(e^{tA})_{t\ge 0}$ be a $C_0$-semigroup on a complex Banach lattice $E$, satisfying the~\eqref{eq:LEP} property with respect to a sequence $(P_n)_{n\in\NN}$ of positive operators on $E$ increasing strongly to the identity. If the semigroup is uniformly bounded, i.e.\ $\sup_{t\ge 0}\|e^{tA}\|<\infty$, and its generator $A$ has compact resolvent, then~\eqref{eq:asymp-pos} holds for every $x\in E_+$.
\end{corollary}
\begin{proof}
	If the generator $A$ has compact resolvent, then the orbits of the semigroup are relatively compact, see~\cite[Corollary V.2.15(i)]{EN00}. The conclusion then follows from Theorem~\ref{thm:asymp-pos}.
\end{proof}

We have shown above that relative compactness of orbits is a sufficient condition for a semigroup satisfying~\eqref{eq:LEP} to be individually asymptotically positive. In that case, the generator of the semigroup automatically has interesting spectral properties. Recall that the \emph{peripheral spectrum} of $A$ is the part of $\sigma(A)$ lying on the line $\{\lambda\in\CC : \Re\lambda = s(A)\}$, and the \emph{peripheral point spectrum} is the intersection of the peripheral spectrum with the point spectrum $\sigma_p(A)$. A subset $S\subset\CC$ is called \emph{cyclic} if whenever $a+ib \in S$ (for some $a,b\in\RR$), then $a+inb\in S$ for all $n\in\ZZ$.
\begin{corollary}
	\label{cor:cyclic}
	Let $(e^{tA})_{t\ge 0}$ be a $C_0$-semigroup on a complex Banach lattice $E$ with $s(A)>-\infty$. If~\eqref{eq:LEP} is satisfied with respect to a sequence $(P_n)_{n\in\NN}$ of positive operators on $E$ increasing strongly to the identity, and the rescaled semigroup $(e^{t(A-s(A))})_{t\ge 0}$ has relatively compact orbits, then the peripheral point spectrum of the generator $A$ is cyclic.
\end{corollary}
\begin{proof}
	Since the rescaled semigroup $(e^{t(A-s(A))})_{t\ge 0}$ has relatively compact orbits, in particular it is bounded. From this observation together with Theorem~\ref{thm:asymp-pos}, we find that the semigroup $(e^{tA})_{t\ge 0}$ is individually asymptotically positive as defined in~\cite[Definition 8.1]{DGK2}. In this case, the cyclicity of the peripheral point spectrum of $A$ follows from~\cite[Theorem 6.3.2]{GTh}.
\end{proof}
Let us remark that the proof in reference~\cite{GTh} of cyclicity of the peripheral point spectrum for individually asymptotically positive semigroups uses the Jacobs-deLeeuw-Glicksberg decomposition, which is also an important technical tool in our analysis in Subsection~\ref{sec:LEP-strong}.

We close the current subsection with some technical remarks concerning Definition~\ref{def:abstract-LEP}. At first glance, the definition of the strong~\eqref{eq:LEP} property appears to be rather complicated. A notion of strong positivity that exists currently in the literature is as follows: if $u\in E_+$ is a quasi-interior point, we say that a vector $v\in E$ is \emph{strongly positive with respect to} $u$, denoted by $v\gg_u 0$, if there exists a constant $c>0$ such that $v\ge cu$. Following~\cite[Definition 2.4]{Ar21}, we could say that a $C_0$-semigroup $(e^{tA})_{t\ge 0}$ has the strong~\eqref{eq:LEP} property with respect to the sequence $(P_n)_{n\in\NN}$ and the quasi-interior point $u$ if, for every $n\in\NN$ and $f\in E_+\setminus\{0\}$, there exists $\tau>0$ such that
\begin{equation}
	\label{eq:strong-LEP-alt}
	P_n e^{tA}f \gg_{P_n u} 0 \quad\text{for all }t\ge\tau.
\end{equation}
While the implied constant $c>0$ may depend on $n$, it is assumed to be independent of $t\ge\tau$. Concrete examples in~\cite[Section 5]{Ar21} show that this type of strong positivity does occur. However there are also natural examples where~\eqref{eq:strong-LEP-alt} cannot be satisfied.
\begin{example}
	Consider the biharmonic heat equation $u_t + (-\Delta)^2 u=0$ on $\RR^N$ with initial data $u_0\ge 0$, $u_0\not\equiv 0$, and for simplicity assume that $u_0$ is continuous and compactly supported. In this case, it was shown in~\cite{GG-lep} that for every compact subset $K\subset\RR^N$, there exists $\tau>0$ such that the corresponding solution $u(t,\cdot)$ satisfies $u(t,x)>0$ on $K$ for all $t\ge\tau$. The solution may be explicitly represented via an inverse Fourier transform
	\begin{equation*}
		u(t,x) = \frac{1}{(2\pi)^{N/2}}\int_{\RR^N} e^{-t|\xi|^4}\hat{u}_0(\xi)e^{ix\cdot\xi} \,d\xi, \qquad t>0, x\in\RR^N,
	\end{equation*}
	where $\hat{u}_0$ is the Fourier transform of $u_0$. A simple computation then yields the estimate
	\begin{equation*}
		\|u(t,\cdot)\|_\infty \le C t^{-N/4} \|u_0\|_{L^1(\RR^N)}\qquad\text{for all }t>0,
	\end{equation*}
	for some constant $C>0$. Thus $\|u(t,\cdot)\|_\infty \to 0$ as $t\to\infty$, which shows that the solution $u(t,\cdot)$ cannot dominate a strictly positive multiple of a positive function on any compact set as $t\to\infty$. Hence~\eqref{eq:strong-LEP-alt} does not hold. Intuitively, this is because $0$, the spectral bound of the operator $(-\Delta)^2$ on $\RR^N$, is not an eigenvalue.
\end{example}

\begin{remark}
	\label{rmk:quasi-int}
	In the setting of~\eqref{eq:strong-LEP-alt}, we have $P_n u>0$, i.e. $P_n u$ is a positive \emph{non-zero} vector, for all $n\in\NN$. Indeed, if $u\in E_+$ is a quasi-interior point that belongs to the kernel of a positive operator $S$ on $E$, then $S=0$. This is a straightforward consequence of the following useful characterisation~\cite[II Theorem 6.3]{Sch}: $u$ is a quasi-interior point of a Banach lattice $E$ if and only if for each $x\in E_+$, the sequence $x_n := x\wedge nu$ converges to $x$ in norm.
\end{remark}
Let $S$ be a positive operator on the Banach lattice $E$, and suppose $u$ is a quasi-interior point. If $F$ is the closed sublattice of $E$ generated by the image $S(E)$ and we view $S$ as an operator from $E$ to $F$, then~\cite[II Proposition 6.4]{Sch} shows that $Su$ is quasi-interior in $F$. If $v\in E_+$ is such that $Sv\ge cSu$ for some constant $c>0$, then $E_{Su}\subset E_{Sv}$ follows from the definition of principal ideals, and we deduce that $Sv$ is quasi-interior in $F$ as well. We have thus proved the following relationship between the different notions of local eventual positivity.
\begin{proposition}
	Let $\mathcal{T}=(e^{tA})_{t\ge 0}$ be a $C_0$-semigroup on the complex Banach lattice $E$. Let there be given a sequence of positive operators $(P_n)_{n\in\NN}$ on $E$ increasing strongly to the identity, and a quasi-interior point $u\in E_+$, such that $\mathcal{T}$ has the strong LEP property with respect to $(P_n)_{n\in\NN}$ and $u$ as defined by~\eqref{eq:strong-LEP-alt}. Then $\mathcal{T}$ has the strong~\eqref{eq:LEP} property with respect to $(P_n)_{n\in\NN}$ in the sense of Definition~\ref{def:abstract-LEP}.
\end{proposition}

\subsection{Local eventual strong positivity}
\label{sec:LEP-strong}
In this section, we present a result on the triviality of the peripheral point spectrum in the context semigroups that are strongly locally eventually positive in the sense of Definition~\ref{def:abstract-LEP}.
\begin{theorem}
\label{thm:strong-pos}
	Let $E$ be a complex Banach lattice such that there exists an increasing sequence $(Q_n)_{n\in\NN}$ of lattice homomorphisms converging strongly to the identity. Let $(e^{tA})_{t\ge 0}$ be a real $C_0$-semigroup on $E$ such that $s(A)>-\infty$ and the rescaled semigroup $(e^{t(A-s(A))})_{t\ge 0}$ has relatively weakly compact orbits. If the semigroup $(e^{tA})_{t\ge 0}$ satisfies the \emph{strong}~\eqref{eq:LEP} property with respect to $(Q_n)_{n\in\NN}$, then the peripheral point spectrum of the generator $A$ is contained in $\{s(A)\}$.
\end{theorem}
The overall strategy of the proof is similar to~\cite[Theorem 6.3.2]{GTh}. In our setting, however, the positivity properties of the semigroup are not `global', and thus some additional technical problems arise.

\subsection*{Jacobs-deLeeuw-Glicksberg decomposition}
\label{subsec:JdLG}
To prove Theorem~\ref{thm:strong-pos}, we use the Jacobs-deLeeuw-Glicksberg decomposition, which we review here briefly for the convenience of the reader. Firstly, if a semigroup $\mathcal{T}\subset\mathcal{L}(E)$ has relatively weakly compact orbits, then the semigroup operators form a relatively compact subset of $\mathcal{L}(E)$ in the weak operator topology (WOT); see~\cite[Section 2.4, Lemma 4.2]{Kr}. Let $\mathcal{S}$ denote the WOT-closure of $\mathcal{T}$. Then $\mathcal{S}$ is a compact, semitopological semigroup, and it is also abelian if $\mathcal{T}$ is abelian. Now we introduce the \emph{Sushkevich kernel}
\begin{equation}
	\mathcal{K} := \bigcap_{T\in\mathcal{S}} T\mathcal{S}.
\end{equation}
It is not difficult to show that this is a non-empty subset of $\mathcal{S}$, and is the intersection of all ideals in $\mathcal{S}$ (here, of course, we use the algebraic definition of an ideal). Consequently $\mathcal{K}$ is the unique minimal ideal of $\mathcal{S}$. Moreover, $\mathcal{K}$ is even a compact group, and its unit $P$ is a projection onto the closed linear span of the eigenvectors corresponding to the unimodular eigenvalues.

In the proofs below, we will assume without loss of generality that $s(A)=0$, and thus speak of $e^{tA}$ in place of $e^{t(A-s(A))}$. The semigroup $\mathcal{T}=\{e^{tA}\}_{t\ge 0}$ is of course abelian, and one checks that
\begin{equation}
\label{eq:K-kernel}
	\mathcal{K} = \bigcap_{T\in\mathcal{S}} T\mathcal{S} \subseteq \bigcap_{t\ge 0} \overline{\{e^{(t+s)A}:s \ge 0\}}^{\text{WOT}}.
\end{equation}
The unit element $P\in\mathcal{K}$ is the projection onto the closed linear span of the eigenvectors corresponding to the peripheral point spectrum $\sigma_p(A)\cap i\RR$. The range $PE$ is invariant under the semigroup $\mathcal{T}$ since $P$ commutes with each semigroup operator, and the restriction of the semigroup to $PE$ extends to a $C_0$-group, where for each $t>0$, the inverse $(e^{tA}|_{PE})^{-1}$ is given by ${R_t}|_{PE}$ for some $R_t\in\mathcal{K}$.  All the results just quoted can be found in~\cite[Section 2.4]{Kr}, where the general abstract theory is presented, or in~\cite[Section V.2]{EN00} where the presentation is adapted to operator semigroups.

\subsection*{Proof of Theorem~\ref{thm:strong-pos}}
We are now ready to give the proof of the main result. As a first step, we show that the Sushkevich kernel $\mathcal{K}$ comprises of positive operators. This is already known in the case of individually asymptotically positive semigroups, c.f.\ the argument in~\cite[p.\ 94]{GTh}. Note that the following lemma does not require the strong~\eqref{eq:LEP} property.
\begin{lemma}
\label{lem:positive-K}
	Let $E$ be a complex Banach lattice such that there exists an increasing sequence $(Q_n)_{n\in\NN}$ of positive operators converging strongly to the identity. Let $(e^{tA})_{t\ge 0}$ be a real $C_0$-semigroup on $E$ with relatively weakly compact orbits, and assume that the semigroup has the~\eqref{eq:LEP} property. Then the Sushkevich kernel, as defined in~\eqref{eq:K-kernel}, consists of positive operators.
\end{lemma}
\begin{proof}
Let $R\in\mathcal{K}$ be arbitrary. The inclusion in~\eqref{eq:K-kernel} shows that there exists a subnet of operators $(e^{t_j A})_{j\in J}$ (where $J$ is the $R$-dependent indexing set) such that $e^{t_j A} \overset{j}{\longrightarrow} R$ in the weak operator topology. Since the semigroup is real, so is the limit operator $R$. Let $f\in E_+\setminus\{0\}$ and $\varphi \in (E')_+\setminus\{0\}$ be arbitrary. Since $Q_n \to I$ strongly, for every $\varepsilon >0$ there is $n_0\in\NN$ such that
\begin{equation*}
	|\braket{\varphi, Q_n Rf}-\braket{\varphi,Rf}| < \varepsilon \quad\text{for all }n \ge n_0.
\end{equation*}
Therefore
\begin{align*}
	\braket{\varphi,Rf}+\varepsilon &> \braket{Q_n' \varphi, Rf} \\
	&= \lim_j \braket{Q_n'\varphi,e^{t_j A}f} = \lim_j \braket{\varphi,Q_n e^{t_j A}f} \ge 0
\end{align*}
for all $n\ge n_0$. As $\varepsilon>0$ was arbitrary, this shows that $\braket{\varphi, Rf}\ge 0$ for all $f\in E_+\setminus\{0\}$ and $\varphi \in (E')_+\setminus\{0\}$, hence $R$ is a positive operator.
\end{proof}

To complete the proof of Theorem~\ref{thm:strong-pos}, we will require more subtle facts about positive operators.
\begin{proof}[Completion of Proof of Theorem~\ref{thm:strong-pos}]
	As before, let $P$ denote the unit of the group $\mathcal{K}$. If $P=0$ (and hence $\mathcal{K}$ is the trivial group), then the peripheral point spectrum is empty and the conclusion follows immediately. Hence assume $P\ne 0$. By Lemma~\ref{lem:positive-K}, $P$ is a positive projection. We have $e^{tA}P\in\mathcal{K}$ for every $t\ge 0$ since $\mathcal{K}$ is an ideal in $\mathcal{S}$. As $P$ commutes with each operator $e^{tA}$, it follows that the semigroup $\mathcal{T}$ restricted to the range $F:=PE$ of $P$ is a \emph{positive} semigroup. As remarked in Subsection~\ref{subsec:JdLG}, the restricted semigroup $\mathcal{T}|_F$ extends to a $C_0$-group, where $(e^{tA}|_F)^{-1}=R_t|_F$ for some $R_t\in\mathcal{K}$. Now we recall the fact that if $T\in\mathcal{L}(E)$ is an invertible operator such that both $T$ and $T^{-1}$ are positive, then $T$ is a \emph{lattice isomorphism} (and the converse is also true) --- see~\cite[Theorem 2.15]{AB}. Therefore $\mathcal{T}|_F$ consists of lattice isomorphisms, which in particular are disjointness preserving operators (recall~\eqref{eq:disjoint-preserve}).
	
	Since $P$ is a positive projection, the range $F$ is a Banach lattice with the order inherited from $E$ and under a norm equivalent to that of $E$. For convenience, we write $\vee$ and $\wedge$ for the supremum and infimum in $E$ respectively. Then the supremum in $F$ is given by $x\vee_F y = P(x\vee y)$, and similarly for the infimum; see~\cite[Chapter III, Proposition 11.5]{Sch} for more details. We now show that $\dim F=1$. Assume for contradiction that $\dim F \ge 2$. Then there exist $0 \le x,y\in F$ such that $x\perp_F y=0$, i.e.\ $x,y$ are positive elements that are disjoint in $F$. In particular this means $P(x\wedge y)=0$. Since $\mathcal{T}|_F$ is a semigroup of positive and disjointness preserving operators, we have $e^{tA}x, e^{tA}y \ge 0$ and $P(e^{tA}x\wedge e^{tA}y)=0$ for all $t\ge 0$. Recall that we have a sequence $(Q_n)_{n\in\NN}$ of lattice homomorphisms increasing to the identity. For every $n\in\NN$, it follows that
	\begin{equation*}
		0 = P(e^{tA}x\wedge e^{tA}y) \ge PQ_n (e^{tA}x\wedge e^{tA}y) = P(Q_n e^{tA}x\wedge Q_n e^{tA}y) \ge 0,
	\end{equation*}
	hence $P(Q_n e^{tA}x\wedge Q_n e^{tA}y)=0$ for all $t\ge 0$. Since the semigroup $\mathcal{T}$ was assumed to have the strong~\eqref{eq:LEP} property, $Q_n e^{tA}x$ and $Q_n e^{tA}y$ are quasi-interior points in the closed sublattice generated by $Q_n E$, and so is $Q_n e^{tA}x\wedge Q_n e^{tA}y$ (see~\cite[Chapter II, Proposition 6.2]{Sch}). However, recalling Remark~\ref{rmk:quasi-int}, this implies $P=0$ on $Q_n E$. Since this is true for each $n\in\NN$, and $Q_n$ increases to the identity on $E$, we find that $P=0$ on $E$, a contradiction. Thus $\dim F=1$ as claimed.
	
	The assertion of Theorem~\ref{thm:strong-pos} is now obtained as follows: the restricted semigroup $\mathcal{T}|_F$ is bounded on the Banach lattice $F=PE$, and consists of lattice homomorphisms. Thus, by a classical result in the theory of positive operator semigroups~\cite[C-III, Theorem 2.10]{AGG}, the entire spectrum (not merely the \emph{point} spectrum) of the generator $A|_F$ is cyclic. However, $F$ is the range of the spectral projection associated to the peripheral point spectrum of $A$, and is one-dimensional. This implies $\sigma_p(A)\cap i\RR=\{0\}$, and the theorem is proved.
\end{proof}

As a straightforward consequence of Theorem~\ref{thm:strong-pos}, we obtain a result on strong convergence of semigroup orbits similar to~\cite[Corollary 2.2]{ArG21} but under slightly different positivity assumptions.
\begin{corollary}
	Let $E$ be a complex Banach lattice such that there exists an increasing sequence $(Q_n)_{n\in\NN}$ of lattice homomorphisms converging strongly to the identity. Let $(e^{tA})_{t\ge 0}$ be a real $C_0$-semigroup on $E$ such that $s(A)>-\infty$. If the semigroup $(e^{tA})_{t\ge 0}$ satisfies the strong~\eqref{eq:LEP} property with respect to $(Q_n)_{n\in\NN}$, then the following assertions are equivalent:
	\begin{enumerate}[\upshape(i)]
		\item The limit $\lim_{t\to\infty} e^{tA}f$ exists for every $f\in E$.
		\item For each $f\in E$, the orbit $\{e^{tA}f : t \ge 0\}$ is relatively (strongly) compact in $E$.
	\end{enumerate}
\end{corollary}
\begin{proof}
	Observe firstly that condition (i) implies (ii). If (ii) holds, then the semigroup orbits are \emph{a fortiori} relatively weakly compact. Under both the stated conditions, the semigroup is bounded, and thus $s(A)\le0$. The Jacobs-deLeeuw-Glicksberg decomposition as described in~\cite[Theorem V.2.14]{EN00} yields $E=E_s \oplus E_r$, where $E_s = \{f\in E: \lim_{t\to\infty}e^{tA}f = 0\}$ and $E_r$ is the closed linear span of eigenvectors corresponding to $\sigma_p(A)\cap i\RR$. The conclusion (ii)$\Rightarrow$(i) is therefore immediate if $s(A)<0$ (and we even have $e^{tA}f \to 0$ as $t\to\infty$ for all $f\in E$). If $s(A)=0$, then Theorem~\ref{thm:strong-pos} shows that the only eigenvalue of $A$ on the imaginary axis is 0, and (i) follows as well.
\end{proof}

\section{Stability of locally eventually positive semigroups}
\label{sec:stability}

The final section of this paper is concerned with the spectral bound of the generator of semigroups satisfying the~\eqref{eq:LEP} property. In the study of stability of semigroups, it is a natural question to ask if the spectral bound $s(A)$ of the generator belongs to the spectrum (provided that $\sigma(A)\ne\emptyset$ of course). It is well-known that this is true in the case that the operator $A$ generates a \emph{positive} semigroup on a Banach lattice (actually, this remains true on more general ordered Banach spaces, see~\cite[Theorem 5.3.1]{ABHN}). Let us briefly recall the proof of this classic result, as it is relevant to the forthcoming Theorem~\ref{thm:spectral-bound}. It is a non-trivial fact that if $A$ generates a positive $C_0$-semigroup on a Banach lattice, then the Laplace transform representation of the resolvent
\begin{equation*}
	R(\lambda,A)f = \int_0^\infty e^{-\lambda t}e^{tA}f \,dt \qquad\text{for all }f\in E
\end{equation*}
is valid as an improper Riemann integral for all $\Re\lambda >s(A)$, which is a larger domain of convergence than can be expected for a general $C_0$-semigroup. We refer to~\cite[Theorem 12.7]{BFR} for a direct proof. Using the positivity of the semigroup, a simple calculation then yields
\begin{equation}
\label{eq:resolvent-est}
	|R(\lambda,A)f| \le R(\Re\lambda,A)|f| \qquad\text{for all }f\in E\text{ and }\Re\lambda >s(A),
\end{equation}
and from the above inequality, it is straightforward to obtain the conclusion $s(A)\in\sigma(A)$---see for example~\cite[Corollary 12.9]{BFR} for the details.

The Laplace transform representation of the resolvent was shown to hold for individually eventually positive $C_0$-semigroups on Banach lattices in~\cite[Proposition 7.1]{DGK1}. In place of inequality~\eqref{eq:resolvent-est}, we have instead the weaker result
\begin{equation}
	\label{eq:resolvent-remainder}
	|R(\lambda,A)f| \le R(\Re\lambda,A)|f| + r_f(\Re\lambda) \qquad\text{for all }\Re\lambda >s(A),
\end{equation}
where $r_f : (s(A),\infty) \to E$ is a `remainder term' that is norm-bounded as $\Re\lambda\downarrow s(A)$ --- see~\cite[Lemma 7.4]{DGK1} for the details. Nevertheless, inequality~\eqref{eq:resolvent-remainder} is sufficient to show that $s(A)\in\sigma(A)$ remains true for individually eventually positive $C_0$-semigroups, see~\cite[Theorem 7.6]{DGK1}. In Theorem~\ref{thm:spectral-bound} below, we will adopt the same strategy to show that $s(A)\in\sigma(A)$ holds for semigroups satisfying the~\eqref{eq:LEP} condition and an additional assumption on the essential spectrum of the generator.

As a first step, we show that~\cite[Theorem 1.5.3]{ABHN} remains valid for eventually positive functions on $[0,\infty)$ with values in an ordered Banach space with \emph{normal positive cone} (see~\cite[Appendix C]{ABHN} for the definitions). Here it suffices to note that all Banach lattices have normal positive cone. For locally Bochner-integrable functions, we use the notation $\abs(u)$ for the \emph{abscissa of convergence} and $\hol(\hat{u})$ for the \emph{abscissa of holomorphy} of the Laplace transform $\hat{u}$, as defined respectively on p.~27 and p.~33 of~\cite{ABHN}.
\begin{lemma}
	\label{lem:hol=abs}
	Let $X$ be an ordered Banach space with normal positive cone. Let $u\in L^1_{\mathrm{loc}}([0,\infty);X)$ satisfy $\abs(u)<\infty$, and suppose that there exists $\tau\ge 0$ such that $u(t)\ge 0$ for a.e.\ $t\ge\tau$. Then $\hol(\hat{u})=\abs(u)$, and if $\abs(u)>-\infty$, then $\hat{u}$ has a singularity at $\abs(u)$.
\end{lemma}
\begin{proof}
	Define $v(t):=u(t+\tau)$ for all $t\ge 0$, so that $v(t)\ge 0$ for a.e.\ $t\ge 0$. By~\cite[Theorem 1.5.3]{ABHN}, it holds that $\hol(\hat{v})=\abs(v)$. On the other hand, a simple computation shows that
	\begin{equation}
		\hat{u}(\lambda) = e^{-\lambda\tau} \hat{v}(\lambda) + \int_0^\tau e^{-\lambda t}u(t)\,dt.
	\end{equation}
	The above formula shows that $\hat{u}$ converges if and only if $\hat{v}$ converges, and likewise $\hat{u}$ and $\hat{v}$ have the same domain of holomorphy. Therefore $\hol(\hat{u})=\hol(\hat{v})=\abs(v)=\abs(u)$.
\end{proof}

An essential ingredient for the proof of Theorem~\ref{thm:spectral-bound} is the following Laplace transform result. It is derived similarly to~\cite[Proposition 6.1]{Ar21}; however, we include a proof since our statement differs from the mentioned reference.
\begin{proposition}
	\label{prop:hol=abs}
	Let $E$ be a complex Banach lattice, and $T:E\to E$ a positive operator. Suppose that $(e^{tA})_{t\ge 0}$ is a $C_0$-semigroup on $E$ such that $[0,\infty) \ni t \mapsto u_f(t):=Te^{tA}f\in E$ is an eventually positive function for every $f\in E_+$. Then, for every $f\in E$, the formula
	\begin{equation}
		\label{eq:resolvent-integral}
		T R(\lambda,A)f = \int_0^\infty e^{-\lambda t}Te^{tA}f\,dt
	\end{equation}
	holds for all $\Re\lambda>\abs(u_f)$, where the integral converges as an improper Riemann integral. Moreover, $\abs(u_f)\in\sigma(A)$ and in particular, $\abs(u_f)\le s(A)$.
\end{proposition}
\begin{proof}
	It is well-known that a positive operator $T:E\to E$ on a Banach lattice $E$ is automatically bounded (see for instance~\cite[Theorem 4.3]{AB}). Thus the formula~\eqref{eq:resolvent-integral} holds for all $\Re\lambda>\omega_0$, where $\omega_0\in\RR$ is the growth bound of the semigroup. As $E_\RR$ spans $E$, and the positive cone of a Banach lattice is generating (i.e.\ $E=E_+-E_+$), therefore it suffices to prove the proposition for $f\in E_+$. 
	
	Given a positive vector $f$, the function $u_f(t)=Te^{tA}f$ is eventually positive by assumption. Since we have $\hol(u_f)=\abs(u_f)$ by Lemma~\ref{lem:hol=abs}, it follows from the identity theorem for analytic functions that the formula~\eqref{eq:resolvent-integral} holds for all $\Re\lambda>\abs(u_f)$. Moreover, by Lemma~\ref{lem:hol=abs} again, $\abs(u_f)$ is a singularity of the Laplace transform $\widehat{u_f}$. Formula~\eqref{eq:resolvent-integral} yields that $\abs(u_f)$ is a singularity of the function $\lambda\mapsto TR(\lambda,A)f$, and since $T$ is a bounded operator, this shows that $R(\cdot,A)f$ has a singularity at $\abs(u_f)$ as well. Therefore $\abs(u_n)\in\sigma(A)$ and $\abs(u_f)\le s(A)$ follows immediately.
\end{proof}

Before stating the main theorem, we recall some further definitions. Let $A$ be a densely-defined closed linear operator on a Banach space $E$. The \emph{nullity} of $A$ is $\mathrm{nul}(A):=\dim\ker(A)$, and the \emph{deficiency} of $A$ is $\mathrm{def}(A):=\codim\rg(A)=\dim(E/\rg A)$. Recall that $A$ is called \emph{Fredholm} if its nullity and deficiency are both finite. Then we can define the \emph{Fredholm domain} $\Delta_F(A)$ and the \emph{essential spectrum} $\sigma_{\mathrm{ess}}(A)$ of $A$ respectively by
\begin{equation}
	\Delta_F(A):=\{\lambda\in\CC: A-\lambda I \text{ is Fredholm}\}, \qquad \sigma_{\mathrm{ess}}(A):= \CC\setminus\Delta_F(A).
\end{equation}
This leads to the natural definition of the \emph{essential spectral bound}:
\begin{equation}
	s_{\mathrm{ess}}(A):=\sup\{\Re\lambda:\lambda\in\sigma_{\mathrm{ess}}(A)\} \in [-\infty,\infty].
\end{equation}
Note that $\sigma_{\mathrm{ess}}(A)\subseteq\sigma(A)$, and hence $s_{\mathrm{ess}}(A)\le s(A)$. By~\cite[Theorem IV.5.28]{Kato}, every isolated spectral value of $A$ with finite algebraic multiplicity belongs to $\Delta_F(A)$.

We can now prove the main result of this section.
\begin{theorem}
	\label{thm:spectral-bound}
	Let $(e^{tA})_{t\ge 0}$ be a $C_0$-semigroup on a complex Banach lattice $E$, satisfying the~\eqref{eq:LEP} property with respect to an increasing sequence $(P_n)_{n\in\NN}$ of positive operators. If $s_{\mathrm{ess}}(A)<s(A)$, then $s(A)\in\sigma(A)$.
\end{theorem}
\begin{proof}
	Assume for contradiction that $s(A)\not\in\sigma(A)$. We fix $n\in\NN$ and let $f\in E_\RR$ be arbitrary. By the~\eqref{eq:LEP} property, we can choose $\tau_n\ge 0$ sufficiently large so that $P_n e^{tA}f^+\ge 0$ and $P_n e^{tA}f^-\ge 0$ for all $t\ge\tau_n$. Then
	\begin{equation*}
		|P_n e^{tA}f|=|P_n e^{tA}(f^+ - f^-)|\le |P_n e^{tA}f^+|+|P_n e^{tA}f^-| = P_n e^{tA}|f|
	\end{equation*}
	for all $t\ge\tau_n$. From the above inequality, we deduce
	\begin{align*}
		\left|\int_0^T e^{-\lambda t}P_n e^{tA}f\,dt\right| &\le \int_0^T e^{-(\Re\lambda)t}|P_n e^{tA}f|\,dt \\
		&\le \int_0^{\tau_n} e^{-(\Re\lambda)t}(|P_n e^{tA}f|-P_n e^{tA}|f|)\,dt + \int_0^T e^{-(\Re\lambda)t}P_n e^{tA}|f|\,dt \\
		&\le P_n \int_0^{\tau_n} e^{-(\Re\lambda)t}(|e^{tA}f|-e^{tA}|f|)\,dt + \int_0^T e^{-(\Re\lambda)t}P_n e^{tA}|f|\,dt.
	\end{align*}
	for all $T>\tau_n$ and $\Re\lambda >\abs(u_n)$, where $u_n(t):=P_n e^{tA}f$. Note that the positivity of $P_n$ was used in the last inequality. If we let $T\to\infty$, then by formula~\eqref{eq:resolvent-integral} we obtain
	\begin{equation}
		\label{eq:Pn-resolvent-est}
		|P_n R(\lambda,A)f| \le P_n R(\Re\lambda,A)|f| + P_n r_{n,f}(\Re\lambda) \qquad\text{for all }\Re\lambda >\abs(u_n),
	\end{equation}
	where $r_{n,f}$ is defined by
	\begin{equation}
		\label{eq:rnf}
		r_{n,f}(\sigma):=\int_0^{\tau_n} e^{-\sigma t}(|e^{tA}f|-e^{tA}|f|)\,dt \qquad\text{for all }\sigma\in\RR.
	\end{equation}

	Proposition~\ref{prop:hol=abs} and the assumption that $s(A)\not\in\sigma(A)$ together imply the strict inequality $\abs(u_n)<s(A)$, and thus $\max\{\abs(u_n), s_{\mathrm{ess}}(A)\}=:\omega<s(A)$. As $\sigma(A)$ is a closed subset of $\CC$, we deduce that there exists $\lambda_0\in\sigma(A)$ of the form $\lambda_0=\delta+i\beta$ where $\delta\in (\omega, s(A))$, $[\delta,\infty)$ belongs to the resolvent set $\varrho(A)$, and $\beta\in\RR\setminus\{0\}$. From the discussion preceding the proof, $\lambda_0$ is an isolated spectral value, and hence it is a pole of the resolvent of some order $k\in\NN$. Let $Q_{-k}$ denote the coefficient of the Laurent expansion of the resolvent about $\lambda_0$.
	
	Now consider $\lambda=\lambda_0+\varepsilon$ for $\varepsilon >0$. We obtain
	\begin{equation}
		\label{eq:lambda-ibeta}
		|(\lambda-\lambda_0)^kP_n R(\lambda,A)f| = \varepsilon^k |P_n R(\lambda,A)f| \le \varepsilon^k P_n[R(\delta+\varepsilon,A)|f|+r_{n,f}(\delta+\varepsilon)]
	\end{equation}
	from inequality~\eqref{eq:Pn-resolvent-est}. As $\varepsilon\downarrow 0$, we have
	\begin{equation*}
		\varepsilon^k P_n R(\lambda,A)f \longrightarrow P_n Q_{-k}f \quad\text{and}\quad \varepsilon^k P_n R(\delta+\varepsilon,A)|f| \longrightarrow 0
	\end{equation*}
	in norm, since $\lambda_0$ is a pole of order $k$, while $\delta=\Re\lambda_0 \in\varrho(A)$. Furthermore, $r_{n,f}(\delta+\varepsilon)$ remains norm-bounded (with $n$ fixed) as $\varepsilon\to 0$. Thus the right-hand side of~\eqref{eq:lambda-ibeta} vanishes as $\varepsilon\downarrow 0$, implying that $P_n Q_{-k}=0$ on $E$, since $f\in E_\RR$ was arbitrary. The above arguments hold for each $n\in\NN$, and consequently we find that $P_n Q_{-k}=0$ for all $n\in\NN$. Recalling that $P_n$ converges strongly to the identity, we conclude $Q_{-k}=0$, a contradiction.
\end{proof}

\begin{remark}
	It would be interesting to see if Theorem~\ref{thm:spectral-bound} remains true without the assumption $s_{\mathrm{ess}}(A)<s(A)$, i.e.\ we impose only the obvious condition $s(A)>-\infty$. We leave this as an open problem.
\end{remark}

Under the assumptions of Theorem~\ref{thm:spectral-bound}, we obtain the following analogue of~\cite[Theorem 7.7]{DGK1} for LEP semigroups.
\begin{theorem}
	Let $(e^{tA})_{t\ge 0}$ be a $C_0$-semigroup on a complex Banach lattice $E$, satisfying the~\eqref{eq:LEP} property with respect to an increasing sequence $(P_n)_{n\in\NN}$ of positive operators. If $s_{\mathrm{ess}}(A)<s(A)$, then the following assertions hold:
	\begin{enumerate}[\upshape (i)]
		\item $s(A)$ is a pole of the resolvent $R(\cdot, A)$ of some order $m\in\NN$, and an eigenvalue of $A$ with a positive eigenvector.
		\item Every pole of the resolvent $R(\cdot, A)$ in the peripheral spectrum $\sigma_{\mathrm{per}}(A)$ has order at most $m$.
	\end{enumerate}
\end{theorem}
\begin{proof}
	(i): We know that $s(A)\in\sigma(A)$ by Theorem~\ref{thm:spectral-bound}. As shown in the discussion preceding the proof of that Theorem, the assumption $s_{\mathrm{ess}}(A)<s(A)$ implies that $s(A)$ is a pole of the resolvent. Moreover, as remarked in~\cite[Corollary 6.2]{Ar21}, for every $n\in\NN$ and $f\in E_+$ it holds that
	\begin{equation}
	\label{eq:asymp-pos-resolvent}
		\lim_{\lambda\downarrow s(A)} (\lambda-s(A)) \dist(P_nR(\lambda,A)f, E_+) = 0;
	\end{equation}
	this follows in a straightforward manner from Proposition~\ref{prop:hol=abs} (c.f.~the calculation in~\cite[Corollary 7.3]{DGK1}).
	
	Let $m\in\NN$ be the pole order of $s(A)$, and denote by $Q_{-m}$ the coefficient of $(\lambda-s(A))^{-m}$ in the Laurent expansion of the resolvent about $s(A)$. For every $n\in\NN$, we have $(\lambda-s(A))^m P_n R(\lambda,A) \to P_n Q_{-m}$ strongly as $\lambda\downarrow s(A)$. Formula~\eqref{eq:asymp-pos-resolvent} then shows that
	\begin{equation*}
		0 = \lim_{\lambda\downarrow s(A)} (\lambda-s(A))^m \dist(P_nR(\lambda,A)f, E_+) = \dist(P_n Q_{-m}f, E_+)
	\end{equation*}
	for every $n\in\NN$ and $f\in E_+$. Therefore $P_n Q_{-m}$ is a positive operator and non-zero for sufficiently large $n$, since $Q_{-m}\ne 0$ and $P_n$ converges strongly to the identity. Letting $n\to\infty$, we find that $Q_{-m}$ is a positive, non-zero operator. It is then known that $s(A)$ is an eigenvalue and $\rg Q_{-m}\ne\{0\}$ consists of eigenvectors to $s(A)$. In addition, by the positivity of $Q_{-m}$, there exists a positive eigenvector to $s(A)$. (Essential facts about the Laurent expansion of the resolvent are collected in~\cite[Remark 2.1]{DGK1}. We note that $Q_{-m}$ here corresponds to the operator $U^{m-1}$ in that reference).
	
	(ii): Suppose that $\lambda_0:=s(A)+i\beta\in\sigma_{\mathrm{per}}(A)$, with $0\ne\beta\in\RR$, is a pole of the resolvent. Let $n\in\NN$ and $f\in E_\RR$ be arbitrary. For $\lambda=\lambda_0+\varepsilon$ with $\varepsilon>0$, we obtain
	\begin{equation*}
		|(\lambda-\lambda_0)^k P_n R(\lambda,A)f| \le \varepsilon^k P_n[R(s(A)+\varepsilon,A)|f|+r_{n,f}(s(A)+\varepsilon)]
	\end{equation*}
	for every $k\in\NN$, as in Formula~\eqref{eq:lambda-ibeta}. If $m$ is the pole order of $s(A)$ and $k>m$, then the right-hand side of the above inequality converges in norm to $0$ for each $n\in\NN$ and $f\in E_\RR$ as $\varepsilon\downarrow 0$. However, we may choose some $f\in E_\RR$ and sufficiently large $n$ such that the left-hand side has a non-zero limit as $\varepsilon\downarrow 0$. Thus if $k$ is the pole order of $\lambda_0$, then $k\le m$.
\end{proof}

\section*{Acknowledgements}
Part of this work was initiated during a pleasant visit to the Bergische Universit\"{a}t Wuppertal and the Technische Universit\"{a}t Dresden in June 2022. Together with the University of Sydney, I am grateful for the financial support of these three institutions. I thank Jochen Gl\"{u}ck in particular for sharing a number of useful ideas and insights, and I also thank Sahiba Arora and Daniel Daners for valuable discussions.

\end{document}